\documentclass[11pt]{article}
\usepackage{amssymb,latexsym}
\usepackage{tikz}
\usepackage{amsmath}
\usepackage{amsthm}
\usepackage{enumerate}
\newtheorem{theorem}{Theorem}[section]
\newtheorem{corollary}[theorem]{Corollary}

\newtheorem{definition}[theorem]{Definition}
\newtheorem{examples}[theorem]{Examples}

\newtheorem{lemma}[theorem]{Lemma}
\newtheorem{proposition}[theorem]{Proposition}
\newtheorem{remark}[theorem]{Remark}

\usepackage{authblk}

\usepackage{tikz}

\begin{document}

\title{On a new extension of the zero-divisor graph (II)}

\author{A. Cherrabi, H. Essannouni, E. Jabbouri, , A. Ouadfel\thanks{Corresponding Author: aliouadfel@gmail.com}}
 \affil{Laboratory of Mathematics, Computing and Applications-Information Security (LabMia-SI)\\
 Faculty of Sciences, Mohammed-V University in Rabat.\\
 Rabat. Morocco.}
\date{}
\maketitle
\begin{abstract}We continue our study of the new extension of zero-divisor graph introduced in \cite{Groupe}. We give a complete characterization for the possible diameters of $\widetilde{\Gamma}(R)$ and $\widetilde{\Gamma}(R[x_1,\dots,x_n])$, we  investigate the relation between the zero-divisor graph, the subgraph of total graph on $Z(R)^{\star}$ and $\widetilde{\Gamma}(R)$ and we present some other properties of $\widetilde{\Gamma}(R)$.
\end{abstract}

\section*{Introduction}
The zero-divisor graph of a commutative ring $R$ with $1\neq 0$ was first introduced by Beck \cite{Beck}, where he was interested in colorings. In his work all elements of the ring were vertices
of the graph and two distinct elements $x$ and $y$ are adjacent if and only if $xy=0$. D.F. Anderson and P.S. Livingston  have defined a graph, $\Gamma(R)$, with vertices in $Z(R)^{\star}=Z(R)\setminus \{0\}$, where $Z(R)$ is the set of non-zero zero-divisors of $R$, and for distinct $x,y \in Z(R)^{\star}$, the vertices x and y
are adjacent if and only if $xy=0$ \cite{AnderLiv}. Also, D. F Anderson and A. Badawi introduced the total graph $T(\Gamma (R))$ of a commutative ring $R$ with all elements of $R$ as vertices and for distinct $x,y \in R$, the vertices $x$ and $y$ are adjacent if and only if $x+y \in Z(R)$ \cite{AnderBada}. $Z^{\star}(\Gamma (R))$ denotes the induced subgraph of $T(\Gamma (R))$ where the vertices are the nonzero zero-divisor of $R$.\\
In \cite{Groupe}, we introduced a new graph, denoted  $\widetilde{\Gamma}(R)$, as the undirected simple graph  whose vertices are the nonzero zero-divisors of $R$ and for distinct $x,y\in Z(R)^{\star}$, $x$ and $y$ are adjacent if and only if $xy=0$ or $x+y\in Z(R)$.  \\
Recall that a simple graph $G=(V,E)$ is connected if there exists a path between any two distinct vertices. A graph of order $0$ or  $1$ is called trivial. For distinct vertices $x$ and $y$ of $G$, the distance $d(x,y)$ is the length of the shortest  path connecting $x$ and $y$; if there is no such path, $d(x,y)=\infty$. The diameter of $G$ is $diam(G)=\sup\{d(x,y)/x,y \in V\, \text{and}\, x\neq y \}$. $G$ is complete if it is connected with diameter one and $K_n$ denote the complete graph with $n$ vertices. A hamiltonian cycle of $G$ is a spanning cycle of $G$. Also, $G$ is said to be hamiltonian if $G$ has a hamiltonian cycle. Basic reference for graph theory is \cite{Dies}.\\
As usual, $T(R)$ denotes the total ring of fractions of $R$, $Nil(R)$ the nilradical of $R$. General reference for commutative ring theory is \cite{A}.\\
In this paper,  we continue our study of the graph $\widetilde{\Gamma}(R)$. In the first section, we completely characterize, in the general case, when $\widetilde{\Gamma}(R)$ and $\widetilde{\Gamma}(R[x_1,\dots,x_n])$ are complete graphs. In section 2, We extend our study of cases where $\widetilde{\Gamma}(R)= \Gamma(R)$ and  $\widetilde{\Gamma}(R)=Z^{\star}(\Gamma (R))$ started in \cite{Groupe} to the general case. The section 3 is devoted to giving some other properties of the graph $\widetilde{\Gamma}(R)$.

\section{When $\widetilde{\Gamma}(R)$ is complete }
In this section, we provide sufficient and necessary conditions for $\widetilde{\Gamma}(R)$ to be complete. This result is a generalization of our result concerning the case where $R$ is finite (cf. theorem 2.4 \cite{Groupe}). We recall that $R\simeq R_1\times\dots\times R_n$, where $ R_1,\dots,R_n$ are non-trivial rings, if and only if there exists $e_1,\dots, e_n\in R^{\star}$ such that $\forall i, e_i$ is idempotent, $\sum\limits_{i=1}^ne_i=1$ and if $i\neq j$, $e_ie_j=0$ (cf. proposition 2.1.1 \cite{Hazew}).  Also, it is clear that $Z(T(R))=\{\frac{x}{s}/x\in Z(R), s\in R\setminus Z(R)\}$ and if $R$  is an integral domain so $T(R)$ is a field and thus $\widetilde{\Gamma}(R)$ and $\widetilde{\Gamma}(T(R))$ are trivial.  We agree that a trivial graph is complete.
\begin{proposition}$\widetilde{\Gamma}(R)$ is complete if and only if $\widetilde{\Gamma}(T(R))$ is complete.
\end{proposition}

\begin{proof}Suppose that $\widetilde{\Gamma}(R)$ is not trivial.\\
$\Rightarrow$): Let $\frac{x}{s},\frac{y}{t}\in Z(T(R))^{\star}$ such that $\frac{x}{s}\neq \frac{y}{t}$ so $x,y\in Z(R)^{\star}$ hence $tx,sy\in Z(R)^{\star}$ and $tx\neq sy$ therefore $tx$ and $sy$ are adjacent in $\widetilde{\Gamma}(R)$ because $\widetilde{\Gamma}(R)$ is complete. If $tx.sy=0$ so $xy=0$ hence $\frac{x}{s}.\frac{y}{t}=0$. If $tx+sy\in Z(R)$ so $\frac{x}{s}+\frac{y}{t}=\frac{tx+sy}{ts}\in Z(T(R))$.\\
$\Leftarrow$): Let $x,y\in Z(R)^{\star}$ such that $x\neq y$ then $\frac{x}{1},\frac{y}{1}\in Z(T(R))^{\star}$ and
$\frac{x}{1}\neq \frac{y}{1}$ so $\frac{x}{1}$ and $\frac{y}{1}$ are adjacent in $\widetilde{\Gamma}(T(R))$. If
$\frac{x}{1} \frac{y}{1}=0$ so $xy=0$ and if $\frac{x}{1}+ \frac{y}{1}\in Z(T(R)$ so $x+y\in Z(R)$.
\end{proof}

\begin{remark}Since diameters of $\widetilde{\Gamma}(R)$ and $\widetilde{\Gamma}(T(R))$ are at most 2 (cf. theorem 2.1 \cite{Groupe}), then  $diam\,(\widetilde{\Gamma}(T(R)))=diam\,(\widetilde{\Gamma}(R))$.
\end{remark}

Below, we give some useful lemmas  to prove our first main result.

\begin{lemma}Let $R\simeq R_1\times\dots\times R_n$, where $n\geq 3$ and  $R_1,\dots,R_n$ are nontrivial rings. If $\widetilde{\Gamma}(R)$ is complete then $R$ is boolean.
\end{lemma}

\begin{proof}Let $x_1\in R_1$ so  $a=(x_1,1,0,\dots, 0)\in Z(R)^{\star}$, also $b=(1-x_1,0,1,\dots, 1)\in Z(R)^{\star}$; we have $a\neq b$ and $a+b=(1,\dots, 1)\notin Z(R)$ so $0=ab$ because $\widetilde{\Gamma}(R)$ is complete hence $x_1^2=x_1$. In the same way, $R_2,\dots, R_n$ are boolean and thus $R$ is boolean.
\end{proof}

\begin{lemma}If $R$ is (up to isomorphism) a subring of a product of two integral domains, then $\widetilde{\Gamma}(R)$ is complete.
\end{lemma}

\begin{proof}We can suppose that $R$ is a subring of $R_1\times R_2$, where $R_1,R_2$ are integral domains, and that  $\widetilde{\Gamma}(R)$ is not trivial. Then  there exists $a,b\in Z(R)^{\star}$ such that $ab=0$. Since $Z(R)$ is  a subset of $ Z(R_1\times R_2)=(R_1\times\{0\})\cup (\{0\}\times R_2)$, we can suppose that $a\in R_1\times\{0\}$ and $b\in \{0\}\times R_2$. Let $x,y\in Z(R)^{\star}$ such that $x\neq y$. If $x,y\in R_1\times\{0\}$ (resp. $x,y\in \{0\}\times R_2$) then $(x+y)b=0$ (resp. $(x+y)a=0$) therefore $x+y \in Z(R)$. If $x\in R_1\times\{0\}$ and $y\in \{0\}\times R_2$ (resp. $x\in \{0\}\times R_2$ and $y\in R_1\times\{0\}$) then $xy=0$.
\end{proof}

\begin{definition} \cite{Hazew}  $R$  is said to be indecomposable if  $R$ cannot
be decomposed into a direct product of two nonzero rings. Otherwise, $R$ is said to be decomposable.
\end{definition}

\begin{lemma}(lemma 2.4.9,\cite{Hazew}) $R$ is an indecomposable ring if and only if it has no nontrivial idempotents.
\end{lemma}

\begin{theorem} $\widetilde{\Gamma}(R)$ is complete if and only if $R$ is boolean or $Z(R)$ is an ideal of $R$ or $R$ is  (up to isomorphism) a subring of a product of two integral domains.
\end{theorem}

\begin{proof}$\Leftarrow$):  If $R$ is boolean then $\widetilde{\Gamma}(R)$ is complete (cf.  the proof of the theorem 2.4 \cite{Groupe}). Also, it is obvious that if $Z(R)$ is an ideal of $R$ then $\widetilde{\Gamma}(R)$ is complete. If $R$ is a subring of a product of two integral domains, it follows from the lemma 1.4 that $\widetilde{\Gamma}(R)$ is complete.\\
$\Rightarrow$): Suppose that $Z(R)$ is not an ideal of $R$ so there exists $a,b\in Z(R)^{\star}$ such that $a+b\notin Z(R)$ so $a\neq b$ and since $\widetilde{\Gamma}(R)$ is complete, $ab=0$.
Let $s=a+b$, $e=\frac{a}{s}$ so $e\in Z(T(R))^{\star}$,  $1-e=\frac{b}{s}$ and $e(1-e)=0$ because $ab=0$ then $e$ is a nontrivial idempotent in $T(R)$ and thus $T(R)=A_1\times A_2$, where $A_1,A_2$ are nontrivial rings.\\
We claim that if $A_1$ or $A_2$ contains a nontrivial idempotent, then $R$ is boolean, indeed, if there exists a nontrivial idempotent in $A_1$ (the other case is similar) then $A_1$ is a product of two rings. Since  $\widetilde{\Gamma}(R)$ is complete, it follows from the proposition 1.1  that $\widetilde{\Gamma}(T(R))$ is complete thus according to the lemma 1.3, $T(R)$ is boolean and thus $R$ is boolean.  \\
We claim that if the only idempotents of $A_1$ and $A_2$ are $0$ and $1$ then $A_1$ and $A_2$ are integral domains, indeed, let $x_1\in Z(A_1)$, we consider $x=(x_1,1),y=(1-x_1,0)$, we have $x,y\in Z(T(R))^{\star}$ and $x\neq y$. Since $\widetilde{\Gamma}(T(R))$ is complete and $x+y=(1,1)\notin Z(T(R))$, $xy=0$ so $x_1^2=x_1$ hence $x_1=0$. In the same way, we show that $A_2$ is an integral domain and thus $R$ is a subring of a product of two integral domains.
\end{proof}

In the rest of this section, we will be interested in the case where the ring is a ring of polynomials. Let's start by recalling McCoy's theorem.
\begin{theorem}[McCoy's Theorem, \cite{Huck}]a polynomial $p(x)\in R[x]$ is a zero-divisor in $R[x]$ if and only if there is a nonzero element $r$ of $R$ such that $rp(x) = 0$.
 \end{theorem}

We recall that  a ring $R$ satisfies property A if each finitely generated ideal $I\subset Z(R)$ has nonzero annihilator (cf. \cite{Huck}, p. 4). As a first step we prove the following lemma.

 \begin{lemma}the following statements are equivalent:
\begin{enumerate}[(a)]
  \item $Z(R[x])$ is an ideal of $R[x]$.
  \item $R$ satisfies property A and $Z(R)$ is an ideal of $R$.
  \item If $I$ is an ideal of $R$ generated by a finite number of zero-divisors, then  $ann_R(I)\neq (0)$.
\end{enumerate}
 \end{lemma}

 \begin{proof}According to theorem 3.3  \cite{Lucas}, (a) and (b) are equivalent. $(a)\Rightarrow (c))$: Suppose that $Z(R[x])$ is an ideal of $R[x]$. Let $I=(a_1,\dots,a_n)$, where $a_1,\dots,a_n\in Z(R)^{\star}$, then $a_1X+\dots+a_nX^n\in Z(R[X])$ so, according to McCoy's theorem, there exist $b\in Z(R)^{\star}$ such that $b(a_1X+\dots+a_nX^n)=0$ therefore $\forall i=1,\dots,n$, $ba_i=0$ hence $ann(I)\neq (0)$.\\
 $(c)\Rightarrow (b)$): Suppose that for every ideal $I$ of $R$ generated by a finite number of zero-divisors, $ann_R(I)\neq (0)$. It is obvious that  $R$ satisfies property A. Let $a,b\in Z(R)^{\star}$, we consider $I=(a,b)$ then there exists $c\in R ^{\star}$ such that $ca=cb=0$ so $c(a+b)=0$ hence $a+b\in Z(R)$.
 \end{proof}

\begin{theorem}$\widetilde{\Gamma}(R[x])$ is complete if and only if  $R$ is  (up to isomorphism) a subring of a product of two integral domains or $R$ satisfies one of the equivalent statements of lemma 1.9.
\end{theorem}

\begin{proof}$\Rightarrow)$: Since $x^2\neq x$, $R[x]$ is not boolean so, according to the theorem 1.7, $R[x]$ is a subring of a product of two integral domains or $Z(R[x])$ is an ideal of $R[x]$. If $R[x]$ is a subring of a product of two integral domains, then $R$ is too.\\
$\Leftarrow)$: If $R$ is a subring of $R_1\times R_2$, where $R_1,R_2$ are integral domains, then it is obvious that $R[x]$ is a subring of $R_1[x]\times R_2[x]$. Since $R_1[x], R_2[x]$ are integral domains then, according to theorem 1.7, $\widetilde{\Gamma}(R[x])$ is complete. If  $Z(R[X])$ is an ideal of $R[x]$ so $\widetilde{\Gamma}(R[x])$ is complete.
\end{proof}

\begin{lemma}$\widetilde{\Gamma}(R[x,y])$ is complete if and only if $\widetilde{\Gamma}(R[x])$ is complete.
\end{lemma}

\begin{proof}$\Rightarrow)$: is an immediate consequence of theorem 1.10\\
$\Leftarrow)$: Since $R[x]$ is not boolean so, according to theorem 1.7, $R[x]$ is isomorph to a subring of a product of two integral domains or $Z(R[x])$ is an ideal of $R[x]$. Since $R[x]$ satisfied property A, then, according to theorem 1.10, $\widetilde{\Gamma}(R[x,y])$ is complete.
\end{proof}

Thus, For a ring of polynomials in $n$ indeterminates, we obtain:

\begin{theorem}$\widetilde{\Gamma}(R[x_1,\dots,x_n])$ is complete if and only if $R$ is  (up to isomorphism) a subring of a product of two integral domains or $R$ satisfies one of the equivalent statements of lemma 1.9.
\end{theorem}

\begin{proof}According to lemma 1.11, $\widetilde{\Gamma}(R[x_1,\dots,x_n])$ is complete if and only if $\widetilde{\Gamma}(R[x_1])$ is complete then the result follows from the theorem 1.10.
\end{proof}

\begin{remark}Let $R$ a ring such that $R$ is not isomorph to a subring of two integral domains, in particular if $R$ is a non-reduced ring, then $\widetilde{\Gamma}(R[x_1,\dots,x_{n+1}])$ is complete if and only if $Z^{\star}(\Gamma (R[x_1]))$ is complete.
\end{remark}

\section{Relation between $\Gamma(R), Z^{\star}(\Gamma (R))$ and $\widetilde{\Gamma}(R)$}

In this section, we will be interested in  the cases where  $\Gamma(R)$ or $Z^{\star}(\Gamma (R))$ coincides with $\widetilde{\Gamma}(R)$ and we provide sufficient and necessary conditions concerning these situations. The results are generalizations of our results concerning the finite case processed in \cite{Groupe}.\\
As first step, we shall prove the following lemmas.

\begin{lemma}$\widetilde{\Gamma}(R)=\Gamma(R)$ if and only if $\widetilde{\Gamma}(T(R))=\Gamma(T(R))$.
\end{lemma}

\begin{proof}$\Rightarrow)$: Suppose that $\widetilde{\Gamma}(R)=\Gamma(R)$.  Since $\Gamma(T(R))$ is a spanning subgraph of $\widetilde{\Gamma}(T(R))$, if two vertices of $\Gamma(T(R))$ are adjacent then they are adjacent as vertices of $\widetilde{\Gamma}(T(R))$. On the other hand, let $\frac{x}{s},\frac{y}{t}\in Z(T(R))^{\star}$ such that $\frac{x}{s}$ and $\frac{y}{t}$ are adjacent in $\widetilde{\Gamma}(T(R))$. we suppose that $\frac{x}{s}+\frac{y}{t}\in Z(T(R))$, if not  $\frac{x}{s}\frac{y}{t}=0$ hence  $\frac{x}{s}$ and $\frac{y}{t}$ are adjacent in $\Gamma(T(R))$, so $tx+sy\in Z(R)$ hence $tx$ and $sy$ are adjacent in $\widetilde{\Gamma}(R)$ therefore $tx.sy=0$ because $\widetilde{\Gamma}(R)=\Gamma(R)$ so $xy=0$ and thus $\frac{x}{s}\frac{y}{t}=0$.\\
 $\Leftarrow)$: Since $\Gamma(R)$ is a spanning subgraph of $\widetilde{\Gamma}(R)$, if two vertices of $\Gamma(R)$ are adjacent then they are adjacent as vertices of $\widetilde{\Gamma}(R)$. Let $x,y\in Z(R)^{\star}$ such that $x$ and $y$ are adjacent in $\widetilde{\Gamma}(R)$. If $x.y=0$ then $x$ and $y$ are adjacent in $\Gamma(R)$; if not $x+y\in Z(R)$ so $\frac{x}{1}+\frac{y}{1}\in Z(T(R))^{\star}$ then $\frac{x}{1}$ and $\frac{y}{1}$ are adjacent in $\widetilde{\Gamma}(T(R))$ hence $\frac{x}{1}\frac{y}{1}=0$ because $\widetilde{\Gamma}(T(R))=\Gamma(T(R))$ and thus $x$ and $y$ are adjacent in $\Gamma(R)$.
\end{proof}

\begin{lemma}Let $R$ a decomposable ring. If $\widetilde{\Gamma}(R)=\Gamma(R)$, then $R\simeq \mathbb{Z}_2^2$.
\end{lemma}

\begin{proof}Since $R$ is decomposable, there exists two non-trivial rings $R_1,R_2$  such that $R\simeq R_1\times R_2$. We claim that $R_1\simeq \mathbb{Z}_2$, indeed,  suppose that there exists $x_1\in R_1$ such that $x_1\notin \{0,1\}$ so we consider $x=(x_1,0),y=(1,0)$ then $x,y\in Z(R)^{\star}$, $x\neq y$ because $x_1\neq 1$. We have $x+y\in Z(R)$ and since $\widetilde{\Gamma}(R)=\Gamma(R)$, $xy=0$ then $x_1=0$. In the same way, $R_2\simeq \mathbb{Z}_2$ and thus $R\simeq \mathbb{Z}_2^2$.
\end{proof}

\begin{theorem}$\widetilde{\Gamma}(R)=\Gamma(R)$  if and only if $\Gamma(R)$ is complete.
\end{theorem}

\begin{proof} $\Leftarrow)$: It is clear that if $\Gamma(R)$ is complete then $\widetilde{\Gamma}(R)=\Gamma(R)$.\\
$\Rightarrow)$: We can suppose that $R$ is not isomorph to   $\mathbb{Z}_2^2$, if not $\Gamma(R)$ is complete.\\
We claim that $\forall x\in Z(R), x^2=0$, indeed, let $ x\in Z(R)^{\star}$. \\
Case 1: If $2x\neq 0$ so $x\neq -x$ then $x^2=0$ because $x$ and $-x$ are adjacent in $\widetilde{\Gamma}(R)$.\\
Case 2: If $2x=0$, let $y\in Z(R)^{\star}$ such that $xy=0$. \\
We claim that $x+y\in Z(R)$, indeed, if $x+y\notin Z(R)$, $e=\frac{x}{x+y}\in T(R)$ verify $e(1-e)=0$ in $T(R)$ so $e^2=e$ and since $e\notin\{0,1\}$, $T(R)$ is decomposable. On the other hand, since $\widetilde{\Gamma}(R)=\Gamma(R)$, according to lemma 2.1, $\widetilde{\Gamma}(T(R))=\Gamma(T(R))$ hence, using lemma 2.2,  $T(R)\simeq \mathbb{Z}_2^2$ and thus  $R\simeq \mathbb{Z}_2^2$. \\
If $x+y=0$, then $x^2=0$. If $x+y\neq 0$, $x+y\in Z(R)^{\star}$, $x\neq x+y$ and $x+(x+y)=y\in Z(R)$ so $x(x+y)=0$ then $x^2=0$.\\
We know that $Nil(R)\subset Z(R)$ and since $\forall x\in Z(R), x^2=0$, $N(R)=Z(R)$. If $x,y\in Z(R)^{\star}$ such that $x\neq y$ then $x+y\in Z(R)$ and thus $x.y=0$ because $\widetilde{\Gamma}(R)=\Gamma(R)$.
\end{proof}

According to theorem 2.8 \cite{AnderLiv} and the previous theorem, we obtain:

\begin{corollary}$\widetilde{\Gamma}(R)=\Gamma(R)$  if and only either $R\simeq \mathbb{Z}_2^2$ or for all $x,y\in Z(R), xy=0$.
\end{corollary}

Next, we turn to the situation where $\widetilde{\Gamma}(R)=Z^{\star}(\Gamma(R))$. Let's start by proving the following lemma.

\begin{lemma}$\widetilde{\Gamma}(R)=Z^{\star}(\Gamma(R))$ if and only if $\widetilde{\Gamma}(T(R))=Z^{\star}(\Gamma(T(R)))$.
\end{lemma}

\begin{proof}$\Rightarrow)$:  Suppose that $\widetilde{\Gamma}(T(R))\neq Z^{\star}(\Gamma(T(R)))$ so there exists $\frac{x}{s},\frac{y}{t}\in Z(T(R))^{\star}$  such that $\frac{x}{s}$ and $\frac{y}{t}$ are adjacent in $\widetilde{\Gamma}(T(R))$ but not adjacent in  $Z^{\star}(\Gamma(T(R)))$ hence $xy=0$ but $tx+sy\notin Z(R)$. We have $(tx),(sy)\in Z(R)^{\star}$, $(tx) (sy)=0$ and $tx\neq sy$ (because $\frac{x}{s}\neq \frac{y}{t}$)
so $(tx),(sy)$ are adjacent in  $\widetilde{\Gamma}(R)$. However, $(tx),(sy)$ are not adjacent in $Z^{\star}(\Gamma(R))$ and thus $\widetilde{\Gamma}(R)\neq Z^{\star}(\Gamma(R))$.\\
$\Leftarrow)$: Let $x,y\in Z(R)^{\star}$ such that $x$ and $y$ are adjacent in $\widetilde{\Gamma}(R)$. We suppose that $x.y=0$, if not $x+y\in Z(R)$ so $x$ and  $y$ are adjacent in $Z^{\star}(\Gamma(R))$, then  $\frac{x}{1}\frac{y}{1}=0$ thus $\frac{x}{1},\frac{y}{1}$ are adjacent in $\widetilde{\Gamma}(T(R))$ then $\frac{x}{1}+\frac{y}{1}\in Z(T(R))$ because $\widetilde{\Gamma}(T(R))=Z^{\star}(\Gamma(T(R)))$ therefore $x+y\in Z(R)$.
\end{proof}

\begin{theorem}$\widetilde{\Gamma}(R)=Z^{\star}(\Gamma(R))$ if and only if $T(R)$ is indecomposable.
\end{theorem}

\begin{proof}$\Rightarrow)$: Suppose that $T(R)$ is decomposable so there exists a nontrivial idempotent $e\in T(R)$  then $e,1-e\in Z(T(R))^{\star}$ because $e(1-e)=0$.   Also, $e\neq 1-e$, if not $e=e^2=0$, so $e,1-e$ are adjacent in $\widetilde{\Gamma}(T(R))$, however $e$ and $1-e$ are not adjacent in $Z^{\star}(\Gamma(T(R)))$ because $e+(1-e)=1\notin Z(T(R))$ then $\widetilde{\Gamma}(T(R))\neq Z^{\star}(\Gamma(T(R)))$ therefore, according to the lemma 2.5, $\widetilde{\Gamma}(R)\neq Z^{\star}(\Gamma(R))$ .\\
$\Leftarrow)$: Suppose that $\widetilde{\Gamma}(R)\neq Z^{\star}(\Gamma(R))$ so there exists $x,y\in Z(R)^{\star}$ such that $x,y$ are adjacent in $\widetilde{\Gamma}(R)$ but $x,y$ are not adjacent in $Z^{\star}(\Gamma(R))$ so $xy=0$ and $x+y\notin Z(R)$. We consider $e=\frac{x}{x+y}\in T(R)$, we have $e\notin \{0,1\}$ and $e(1-e)=\frac{xy}{(x+y)^2}=0$ so $e=e^2$ then there is a nontrivial idempotent in $T(R)$ thus $T(R)$ is decomposable.
\end{proof}

\begin{examples}\hfill
\begin{itemize}
  \item [--] In \cite{Groupe}, we showed that if $R$ is a finite ring, $\widetilde{\Gamma}(R)=Z^{\star}(\Gamma(R))$ if and only if $Z^{\star}(\Gamma(R))$ is a complete graph (i.e., $Z(R)$ is an ideal of $R$). However, in general, this result is false: it is obvious that if $Z(R)$ is an ideal of $R$,  then $\widetilde{\Gamma}(R)=Z^{\star}(\Gamma(R))$ but the converse is false as shown by the following example: we consider the commutative ring $R=\mathbb{Z}(+)\mathbb{Z}_6$ the idealization of the $\mathbb{Z}$-module $\mathbb{Z}_6$ in $\mathbb{Z}$ (cf. \cite{Huck}). According to theorem 25.3 \cite{Huck}, $(r,x)\in R$ is a zero-divisor of $R$  if and only if $r\in Z(\mathbb{Z})\cup Z(\mathbb{Z}_6)$ so $(r,x)\in R$ is a zero-divisor of $R$ if and only if $r\in 2\mathbb{Z}\cup 3\mathbb{Z}$. We claim that $\widetilde{\Gamma}(R)=Z^{\star}(\Gamma(R))$, indeed, let $a=(n,x),b=(m,y)\in Z(R)^{\star}$ such that $ab=0$ so $nm=0$ then $n=0$ or $m=0$ and thus $a+b\in Z(R)$. On the other hand, $Z(R)$ is not an ideal of $R$, indeed, $(-2,0),(3,0)\in Z(R)$ but $(-2,0)+(3,0)=(1,0)\notin Z(R)$. Also, we note that $\widetilde{\Gamma}(R)=Z^{\star}(\Gamma(R))$  but $\widetilde{\Gamma}(R)$ is not complete.
  \item[--] It is obvious that, in general,  if $T(R)$ is indecomposable, $R$ is  not necessarily local, indeed, $\mathbb{Z}$ is not local but $T(\mathbb{Z})=\mathbb{Q}$ is indecomposable. For an other example with nontrivial graphs, just consider the same ring $R=\mathbb{Z}(+)\mathbb{Z}_6$ thus $T(R)$ is indecomposable and using theorem 25.1 (3) \cite{Huck}, $R$ is not local.
\end{itemize}
\end{examples}

\begin{corollary}\hfill
\begin{enumerate}[(a)]
  \item Let $R$ such that every non-unit is a zero-divisor. Then, $\widetilde{\Gamma}(R)=Z^{\star}(\Gamma(R))$ if and only if $R$ is indecomposable.
  \item Let $R$ such that $\dim R=0$. Then, $\widetilde{\Gamma}(R)=Z^{\star}(\Gamma(R))$ if and only if $R$ is indecomposable.
  \item Let $R$  a noetherian reduced ring.  $\widetilde{\Gamma}(R)=Z^{\star}(\Gamma(R))$ if and only if $Z^{\star}(\Gamma(R))$ is complete (i.e., $Z(R)$ is an ideal).
  \item Let $R$ an artinian ring.  $\widetilde{\Gamma}(R)=Z^{\star}(\Gamma(R))$ if and only if $Z^{\star}(\Gamma(R))$ is complete (i.e., $Z(R)$ is an ideal).
\end{enumerate}
\end{corollary}

\begin{proof} (a) follows immediately  from  $R\simeq T(R)$.\\
 (b) Let $x$  a  non-unit element of $R$ so there exist a minimal prime ideal $\mathfrak{p}$ of $R$ such that $x\in \mathfrak{p}$ because $\dim R=0$. Then, according to the theorem 2.1 \cite{Huck}, there exist $y\notin \mathfrak{p}$ and an integer $n>0$ such that $yx^n=0$. Let $k>0$ the smallest integer such $yx^k=0$ so $(yx^{k-1})x=0$ and since $yx^{k-1}\neq 0$ so $x\in Z(R)$. \\
 (c) Since $R$ is a noetherian reduced ring,  $\bigcap\limits_{i=1}^{n}\mathfrak{p}_i=\{0\}$, where $\mathfrak{p}_1,\dots, \mathfrak{p}_n$ are the minimal prime ideals of $R$. We claim that if $n\geq 2$, then $T(R)$ is decomposable, indeed, suppose that $n\geq 2$ then, according to the prime avoidance lemma, $\mathfrak{p}_1 \nsubseteq \bigcup\limits_{i=2}^{n}\mathfrak{p}_i$ and $\bigcap\limits_{i=2}^{n}\mathfrak{p}_i\nsubseteq \mathfrak{p}_1$. Let $x\in \mathfrak{p}_1$ such the $x\notin \bigcup\limits_{i=2}^{n}\mathfrak{p}_i$ and $y\in \bigcap\limits_{i=2}^{n}$ such that $y\notin \mathfrak{p}_1$ so $x+y\notin \bigcup\limits_{i=1}^{n}\mathfrak{p}_i=Z(R)$ and since $xy\in \bigcap\limits_{i=2}^{n}\mathfrak{p}_i=0$ then $e=\frac{x}{x+y}$ is a non-trivial idempotent in $T(R)$.\\
 (c) Suppose that $R$ is artinian then $\dim R=0$ so, by (b), $\widetilde{\Gamma}(R)=Z^{\star}(\Gamma(R))$ if and only if $R$ is indecomposable. Using Theorem 8.7 \cite{A}, $R\simeq \prod \limits_{i=1}^{n} R_i$, where $R_1,\dots, R_n$ are local rings and thus the result is an immediate consequence.
\end{proof}

\section{Others Properties}
In this section, we give some other properties of $\widetilde{\Gamma}(R)$ .

\begin{theorem} If $|Z(R)^{\star}|\geq 3$, $\widetilde{\Gamma}(R)$ is hypertriangulated, i.e., every edge is in a triangle.
\end{theorem}

\begin{proof} Let $x,y\in Z(R)^{\star}$ such that $x$ and $y$ are adjacent and  $z\in Z(R)^{\star}$ such that $z\notin \{x,y\}$. If $x,y\in N(z)$, where $N(z)$ denotes the set of vertices that are adjacent to $z$,
then $xyz$ is a triangle. Thus we can suppose that $x\notin N(z)$ (by symmetry the other case is similar). Let $a,b,c\in Z(R)^{\star}$ such that $ax=by=cz=0$.\\
Case 1: If $xy=0$: Then $cy(x+z)=0$ so $cy=0$ because $x\notin N(z)$, hence $y(x+c)=0$ therefore $x+c\in Z(R)$. On the other hand, $c\neq x$, if not $xz=cz=0$; $c\neq y$, if not $y(x+z)=0$ then $y=0$. Thus $xyc$ is a triangle.\\
Case 2: If $xy\neq 0$: If $yz=0$, $ay(x+z)=0$ so $ay=0$ because $x\notin N(z)$. Also, $a\neq x$, if not $xy=ay=0$; $a\neq y$, if not $xy=xa=0$. Therefore $xay$ is a triangle. If $yz\neq 0$. Since $ac(x+z)=0$ and $x\notin N(z)$, $ac=0$. There are two case: If $bc=0$, $b(y+c)=0$ so $y+c\in Z(R)$. In this case, $c\neq x$, if not $xz=cz=0$;  $c\neq y$, if not $yz=cz=0$. Therefore $xyc$ is a triangle. If $bc\neq 0$, $a(bc+x)=0$ and $(bc)y=0$. Also, $bc\neq x$, if not $xz=bcz=0$;  $bc\neq y$, if not $yz=bcz=0$. Thus $xy(bc)$ is a triangle.
\end{proof}

\begin{remark}it is clear that the result (a) of theorem 3.3 \cite{Groupe} is  an immediate consequence of the previous theorem.
\end{remark}

We recall that an ideal is called decomposable if it admits a primary decomposition.

\begin{theorem}If the zero ideal is decomposable, in particular if $R$ is a noetherian ring, then there exists $a\in Z(R)^{\star}$ such that for every $x \in Z(R)^{\star}\setminus \{a\}$, $a$ and $x$ are adjacent.
\end{theorem}

\begin{proof}Since $(0)$ is is decomposable, $Z(R)=\bigcup\limits_{i=1}^{n}\mathfrak{p}_i$, where $\mathfrak{p}_i=(0:a_i)$ are the associated ideals to $(0)$. If $n=1$, the result is trivial so  we can suppose that $n\geq 2$. Also, if $a_1\dots a_n\neq 0$, let $a=a_1\dots a_n$ so $a\in Z(R)^{\star}$ and $\forall x\in Z(R)^{\star}$, there exists $i$ such that $x\in \mathfrak{p}_i=(0:a_i)$ so $ax=0$. Suppose that $a_1\dots a_n= 0$, let $X=\{m\in \{1,\dots,n-1\}/\exists \{a_{i_{1}},\dots,a_{i_{m}}\}\subset \{a_{1},\dots,a_{m}\}:a_{i_{1}}.\dots. a_{i_{m}}\neq 0\}$. We have $X\neq \emptyset$ because $a_1\neq 0$. Let $k=\max(X)$ and $a=a_{i_{1}}.\dots. a_{i_{k}}\neq 0$. Then, $\forall x\in Z(R)^{\star}$, there exists $i$ such that $x\in \mathfrak{p}_i=(0:a_i)$.
If $i\in \{i_1,\dots,a_k\}$ then $ax=0$. If $i\notin \{i_1,\dots,a_k\}$ so $aa_i=0$ because $k=\max (X)$  and since $xa_i=0$ then $a_i(x+a)=0$ hence $x+a\in Z(R)$.
\end{proof}

\begin{theorem}If $R$  is a finite ring, then $\widetilde{\Gamma}(R)$ is hamiltonian.
\end{theorem}

\begin{proof}Since $R$ is finite, then $R\simeq R_1\times \dots R_n$, where $R_1,\dots, R_n$ are local rings. We can suppose that $n\geq 2$, if not $\widetilde{\Gamma}(R)$ is complete. Up to isomorphism, $Z(R)=M_1\cup \dots M_n$, where $M_i=R_1\times \dots\ \times \mathfrak{m}_{i}\times \dots \times R_n$ and $\mathfrak{m}_{i}$ is the maximal ideal of $R_i$. Let $X_1=M_1\setminus \{0\}$ and for every $i=2,\dots, n$, $X_i=M_i\setminus(\bigcup \limits_{j=1}^{i-1}M_j)$ then $Z(R)^{\star}=\bigcup \limits_{i=1}^{n}X_i$ and $X_i\cap X_j=\emptyset$ if $i\neq j$.\\
Case where $n=2$: If $\mathfrak{m}_{1}=\{0\}$ and $\mathfrak{m}_{2}=\{0\}$ then $R_1,R_2$ are fields and thus $\widetilde{\Gamma}(R)$ is complete (cf. theorem 2.4 \cite{Groupe}) so we can suppose that $\mathfrak{m}_{1}\neq \{0\}$ and consider $X_1$ as the orderly sequence $x_{11},\dots, x_{1p_{1}}$, with $x_{11}=(0,1)$ and $x_{1p_{1}}=(a,0)$, where $a\in \mathfrak{m}_{1}\setminus \{0\}$. Also, we consider $X_2$ as the orderly sequence $x_{21},x_{2p_{2}}$  with $x_{2p_{2}}=(1,0)$ so $x_{11}\--\dots \--x_{1p_{1}}$ (because $x_{11},\dots, x_{1p_{1}}\in M_1$), $x_{1p_{1}}\--x_{21}\--\dots \--x_{2p_{2}}$ (because they are in  $M_2$) and $x_{2p_{2}}\--x_{11}$ (because $x_{2p_{2}}x_{11}=0$) and thus we obtain the hamiltonian cycle $x_{11}\--\dots \--x_{1p_{1}}\--x_{21}\--\dots \--x_{2p_{2}}\--x_{11}$.\\
Case where $n\geq 3$: If for every $i$, $|R_i|=2$, then $R$ id boolean and thus $\widetilde{\Gamma}(R)$ is complete (cf. theorem 2.4 \cite{Groupe}) so we can suppose that $|R_n|\geq 3$. We consider $X_1$ as the orderly sequence $x_{11},\dots, x_{1p_{1}}$  with $x_{11}=(0,\dots,0,1)$, $x_{1p_{1}}=(0,\dots,0,a)$ and $a\in R_n\setminus\{0,1\}$. Also, for every $i=2,\dots,n$, we consider $X_i$ as the orderly sequence $x_{i1},\dots, x_{ip_{i}}$  with  $x_{ip_{i}}=(1,\dots,1,0,\dots,0)$ ($1$  being repeated (i-1) times) so for $i=1,\dots,n-1$, $x_{ip_{i}}\in X_i\cap M_{i+1}$. Then, $x_{11}\--\dots\-- x_{1p_{1}}$ (because they are in $M_1$) and for every $i=2,\dots,n$, $x_{{i-1}p_{i-1}}\--x_{i1}\--\dots\-- x_{ip_{i}}$ (because they are in the same $M_i$) and $x_{np_{n}}\--x_{11}$ (because $x_{np_{n}}x_{11}=0$) therefore we obtain the hamiltonian cycle
$x_{11}\--\dots\-- x_{1p_{1}}\--x_{21}\--\dots\--x_{2p_{2}}\--x_{31}\--\dots\--x_{np_{n}}\--x_{11}$.
\end{proof}


\end{document}